\documentclass[a4paper,11pt]{article}

\usepackage{geometry}
\usepackage{amsmath,amssymb,amsthm,mathtools}
\usepackage{mathrsfs}
\usepackage{tikz}
\usetikzlibrary{arrows.meta}
\usepackage{subcaption}
\usepackage[colorlinks=true,linkcolor=blue,citecolor=blue,urlcolor=blue]{hyperref}
\usepackage{float}
\numberwithin{equation}{section}

\theoremstyle{plain}
\newtheorem{theorem}{Theorem}[section]
\newtheorem{lemma}[theorem]{Lemma}
\newtheorem{proposition}[theorem]{Proposition}
\numberwithin{equation}{section}

\theoremstyle{remark}
\newtheorem{remark}[theorem]{Remark}

\theoremstyle{definition}
\newtheorem{definition}[theorem]{Definition}

\begin{document}

\title{Low-regularity well-posedness for a mixed-sign quadratic Dirac equation on $N$-star metric graphs}
\author{Huichao Xing$^{1,2}$\quad Zhipeng Yang$^{2,3}$\thanks{Corresponding author:yangzhipeng326@163.com.}\\
{\small $^{1}$ Faculty of Education, Yunnan Normal University, Kunming, China}\\
{\small $^{2}$ Yunnan Key Laboratory of Modern Analytical Mathematics and Applications, Kunming, China.}\\
{\small $^{3}$ Department of Mathematics, Yunnan Normal University, Kunming, China}}
    
\date{}
\maketitle

\begin{abstract}
We study the Cauchy problem for a mixed-sign quadratic Dirac equation on a noncompact $N$--star metric graph $G$,
\[
\mathrm{i}\partial_t \psi = D\psi - \mathcal N(\psi),
\qquad
\psi(0)=\psi_0,
\]
where $\psi=(\psi_1,\psi_2)^{\mathsf T}:\mathbb{R}\times G\to\mathbb{C}^2$ and $D$ denotes the self-adjoint Dirac--Kirchhoff operator on $G$.
The nonlinearity acts edgewise and is given by a bilinear interaction between the positive and negative spectral parts,
\[
\mathcal N(\psi)=\mathcal B\bigl(\Pi_+\psi,\Pi_-\psi\bigr),
\]
where $\Pi_\pm$ are the spectral projections of $D$ and $\mathcal B$ is a fixed bilinear map on $\mathbb{C}^2$ applied componentwise on each edge.
This is a model quadratic interaction tailored to the mixed-sign Bourgain-space mechanism, rather than a general nonlinear Dirac equation on graphs.
Using Bourgain-type spaces associated with the spectral resolution of $D$ and a mixed-sign bilinear estimate on $N$--star graphs, we prove local well-posedness in the operator Sobolev space $H_D^s(G)$ for \(s>-\frac18\).
We also establish a blow-up alternative in $H_D^s(G)$ for the maximal forward lifespan.
\end{abstract}



\section{Introduction and Main Results}

The Dirac equation is a basic dispersive model in relativistic quantum mechanics, describing the time evolution of spin-$\frac12$ fields through a first-order system.
For background we refer to \cite{Thaller1992}.
In $\mathbb{R}^d$ the free flow can be written as
\[
i\partial_t\psi = D_m\psi,
\qquad
D_m = -i\alpha\cdot\nabla + m\beta,
\]
with Dirac matrices $\alpha=(\alpha_1,\dots,\alpha_d)$ and $\beta$.
Nonlinear Dirac dynamics arises when self-interactions are incorporated,
\[
i\partial_t\psi = D_m\psi - \mathcal N(\psi).
\]
Prominent examples are the Soler model \cite{Soler1970},
\[
\mathcal N(\psi)=g(\bar\psi\psi)\,\beta\psi,
\]
and the Gross--Neveu model \cite{GrossNeveu1974},
\[
\mathcal N(\psi)=g(\bar\psi\psi)\,\psi.
\]
The existence and qualitative properties of solitary waves and other nonlinear phenomena have been studied extensively, see for instance \cite{CazenaveVazquez1986}.

In one spatial dimension, quadratic interactions constitute another important class.
Machihara \cite{Machihara2005} introduced an $X^{s,b}$ framework tailored to the one-dimensional Dirac dispersion and obtained local well-posedness at the sharp regularity threshold $s>-\frac18$ for a family of quadratic nonlinearities.
The analysis relies on a mixed-sign bilinear estimate capturing the interaction between the positive and negative spectral components of the Dirac flow.
Related low-regularity results were later developed in \cite{MachiharaNakanishiTsugawa2010}.
More generally, Bourgain's Fourier restriction method \cite{Bourgain1993} underlies the $X^{s,b}$ theory, and an exposition in the context of dispersive PDE is given in \cite{TaoBilinearDispersive}.

Metric graphs, or quantum graphs, model wave propagation on networks by combining one-dimensional edge dynamics with vertex coupling conditions.
They support a broad class of dispersive equations, with the nonlinear Schr\"odinger equation as a prototypical example, and exhibit a strong dependence of dynamics on topology and vertex transmission.
We refer to \cite{BerkolaikoKuchment2013,Kuchment2008} for general background and to \cite{Noja2014,AdamiCacciapuotiFincoNoja2012} for nonlinear Schr\"odinger dynamics on star graphs and related settings.

Dirac operators on metric graphs have been studied from the perspective of self-adjoint realizations, spectral theory, and scattering, see for example \cite{BullaTrenkler1990,BolteHarrison2003}.
In the nonlinear direction, Borrelli, Carlone, and Tentarelli \cite{Borrelli2021} introduced a self-adjoint Dirac operator on noncompact metric graphs with Kirchhoff-type vertex conditions and investigated nonlinear Dirac flows, including a blow-up alternative.

Our goal is not to treat a general nonlinear Dirac equation on graphs, but to transfer to an $N$-star metric graph the specific mixed-sign quadratic mechanism behind Machihara's low-regularity theory on the line.
The vertex condition couples the edgewise dynamics, and the standard space--time Fourier transform on $\mathbb{R}\times\mathbb{R}$
is not available.
We define Bourgain-type spaces through the spectral resolution of the self-adjoint Dirac--Kirchhoff operator $D$ on the star graph
and establish a mixed-sign bilinear estimate for the interaction between the positive and negative spectral components.
The transfer from the one-dimensional Dirac evolution to the star-graph setting relies on the explicit spectral representation of $D$
and on the fact that the vertex coupling acts through a finite-dimensional unitary matrix on the edge index.
Accordingly, the equation considered here should be viewed as a quadratic model problem adapted to the $\Pi_+$/$\Pi_-$ interaction, rather than as a graph analogue of the physically standard Soler or Gross--Neveu equations.

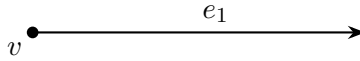
\begin{figure}[H]
\centering
\begin{tikzpicture}[scale=1.1, line cap=round, line join=round, >=Stealth]
  \fill (0,0) circle (2pt);
  \node[below left] at (0,0) {$v$};
  \draw[thick,->] (0,0) -- (4,0);
  \node[above] at (2.2,0) {$e_1$};
\end{tikzpicture}
\caption{The $1$-star metric graph $G$}
\label{fig:N1star}
\end{figure}

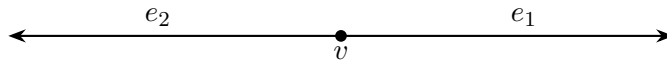
\begin{figure}[H]
\centering
\begin{tikzpicture}[scale=1.1, line cap=round, line join=round, >=Stealth]
  \fill (0,0) circle (2pt);
  \node[below] at (0,0) {$v$};
  \draw[thick,->] (0,0) -- (4,0);
  \draw[thick,->] (0,0) -- (-4,0);
  \node[above] at (2.2,0) {$e_1$};
  \node[above] at (-2.2,0) {$e_2$};
\end{tikzpicture}
\caption{The $2$-star metric graph $G$}
\label{fig:N2star}
\end{figure}

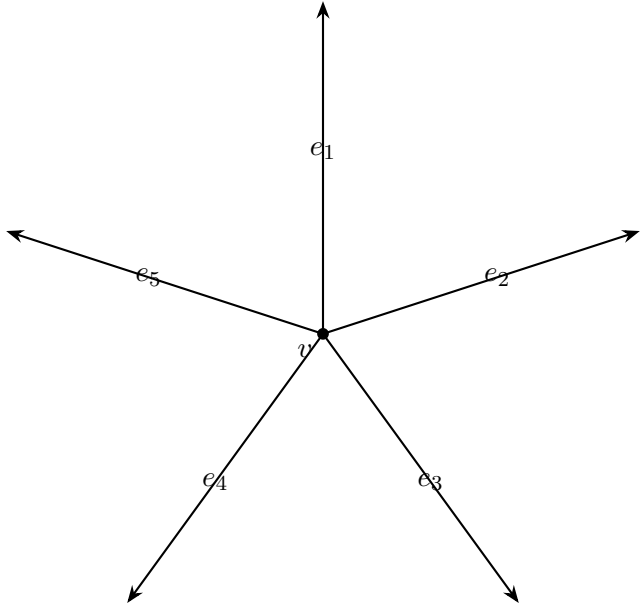
\begin{figure}[H]
\centering
\begin{tikzpicture}[scale=1.1, line cap=round, line join=round, >=Stealth]
  \fill (0,0) circle (2pt);
  \node[below left] at (0,0) {$v$};

  \foreach \ang/\lab in {90/{$e_1$},18/{$e_2$},-54/{$e_3$},-126/{$e_4$},162/{$e_5$}} {
    \draw[thick,->] (0,0) -- (\ang:4);
    \node at (\ang:2.2) {\lab};
  }
\end{tikzpicture}
\caption{The $5$-star metric graph $G$}
\label{fig:N5star}
\end{figure}

Let $G$ be the noncompact $N$-star metric graph obtained by gluing $N$ copies of $\mathbb{R}_+$ at a common vertex; see Figures~\ref{fig:N1star}--\ref{fig:N5star}.
Let $m\ge 0$ and let $D$ be the Dirac--Kirchhoff operator acting edgewise by
\[
(D\psi)_e=\left(-i\sigma_1\partial_x+m\sigma_3\right)\psi_e
\quad\text{on }\mathbb{R}_+^{(e)},
\]
with vertex condition \eqref{eq:DK-vertex} and domain \eqref{eq:D-domain}.
For $s\in\mathbb{R}$ we work in the operator Sobolev space $H_D^s(G)=\mathrm{Dom}(\langle D\rangle^s)$ endowed with the norm \eqref{eq:HDs}.
We study the Cauchy problem
\begin{equation}\label{eq:NLDE}
\begin{cases}
i\partial_t\psi = D\psi - \mathcal N(\psi), & (t,x)\in\mathbb{R}\times G,\\
\psi(0,x)=\psi_0(x), & x\in G,
\end{cases}
\end{equation}
where the nonlinearity acts edgewise and is given by a bilinear interaction between the positive and negative spectral parts,
\[
\mathcal N(\psi)=\mathcal B\bigl(\Pi_+\psi,\Pi_-\psi\bigr).
\]
Throughout the paper, $\mathcal B:\mathbb C^2\times\mathbb C^2\to\mathbb C^2$ is a fixed bilinear map applied componentwise on each edge.
Since the underlying space is finite-dimensional, there exists $C_{\mathcal B}>0$ such that
\[
|\mathcal B(a,b)|\le C_{\mathcal B}|a|\,|b|,
\qquad a,b\in\mathbb C^2.
\]
Typical admissible examples are
\[
\mathcal B_1(a,b)=\begin{pmatrix} a_1b_1\\ a_2b_2\end{pmatrix},
\qquad
\mathcal B_2(a,b)=\begin{pmatrix} a_1b_2\\ a_2b_1\end{pmatrix},
\qquad
\mathcal B_3(a,b)=\langle Ma,b\rangle_{\mathbb C^2}\,v,
\]
with fixed $M\in\mathbb C^{2\times 2}$ and $v\in\mathbb C^2$.
These examples illustrate the class covered by the analysis; no specific physical interpretation is assumed for $\mathcal B$, and the role of the structure $\mathcal B(\Pi_+\psi,\Pi_-\psi)$ is to isolate the mixed-sign interaction for which Machihara's bilinear estimate applies.

\begin{theorem}\label{thm:lwp}
Let $s>-\frac18$ and let $\psi_0\in H_D^s(G)$.
There exist $T=T(\|\psi_0\|_{H_D^s(G)})>0$ and $b\in\left(\frac12,1\right)$ such that \eqref{eq:NLDE} admits a unique solution
\[
\psi\in C([0,T];H_D^s(G))\cap X_T^{s,b}.
\]
The data-to-solution map is locally Lipschitz from $H_D^s(G)$ to $C([0,T];H_D^s(G))$.
Let $[0,T^\ast)$ be the maximal forward lifespan of $\psi$ in $C([0,T);H_D^s(G))$.
If
\[
\sup_{t\uparrow T^\ast}\|\psi(t)\|_{H_D^s(G)}<\infty,
\]
then there exists $\delta>0$ such that $\psi$ extends to a solution on $[0,T^\ast+\delta)$.
\end{theorem}

The condition $s>-\frac18$ is dictated by the sharp mixed-sign bilinear mechanism for the one-dimensional Dirac dispersion in \cite{Machihara2005}.
In the present setting, the bilinear estimate is transferred to $G$ through the spectral representation of $D$ and a finite-dimensional unitary reduction on the edge index.
This yields the same threshold $s>-\frac18$ within the Bourgain-space approach used here.

\begin{remark}\label{rem:limitations}
Theorem~\ref{thm:lwp} concerns only the mixed-sign quadratic structure $\mathcal B(\Pi_+\psi,\Pi_-\psi)$.
The proof does not cover generic polynomial nonlinearities, same-sign quadratic interactions, or gauge-invariant models such as Soler or Gross--Neveu on metric graphs.
Likewise, the threshold $s>-\frac18$ is sharp for the specific mixed-sign Bourgain-space mechanism used here, but we do not claim optimality for broader classes of graph nonlinear Dirac equations.
\end{remark}

The paper is organized as follows.
We recall the spectral setting of the Dirac--Kirchhoff operator and introduce the Bourgain spaces $X^{s,b}$ and their restriction versions $X_T^{s,b}$.
We establish the linear cutoff and Duhamel estimates, prove the mixed-sign bilinear estimate, and complete the contraction argument.
Finally, we prove Theorem~\ref{thm:lwp}, including the blow-up alternative.

\section{Preliminaries and Function Spaces}\label{sec:spaces}

Let $G$ be the noncompact $N$--star metric graph obtained by gluing $N$ copies of $\mathbb{R}_+$ at a common vertex $v$; see Figures~\ref{fig:N1star}--\ref{fig:N5star} for typical examples. We write
\[
G=\bigcup_{e=1}^N \mathbb{R}_+^{(e)},
\qquad
\text{where the endpoints }x=0\text{ on each }\mathbb{R}_+^{(e)}\text{ are identified with }v,
\qquad
x\in(0,\infty)\ \text{on each edge}.
\]
A spinor $\psi=(\varphi,\chi)^{\mathsf T}$ on $G$ is identified with the family
$\{\psi_e\}_{e=1}^N$, where $\psi_e=(\varphi_e,\chi_e)^{\mathsf T}:\mathbb{R}_+\to\mathbb{C}^2$
denotes the restriction to the edge $\mathbb{R}_+^{(e)}$.
Set
\[
L^2(G;\mathbb{C}^2)=\bigoplus_{e=1}^N L^2(\mathbb{R}_+^{(e)};\mathbb{C}^2),
\qquad
\|\psi\|_{L^2(G)}^2=\sum_{e=1}^N\int_0^\infty |\psi_e(x)|^2\,dx,
\]
where $|\cdot|$ is the Euclidean norm on $\mathbb{C}^2$.
Throughout the paper, the inner products on $L^2(G;\mathbb{C}^2)$ and $\mathbb{C}^2$ are taken to be linear in the second argument; in particular,
\[
\langle f,g\rangle_{L^2(G)}
=\sum_{e=1}^N\int_0^\infty \langle f_e(x),g_e(x)\rangle_{\mathbb{C}^2}\,dx,
\qquad
\langle a,b\rangle_{\mathbb{C}^2}=a^\ast b.
\]

We use the Pauli matrices
\[
\sigma_1=
\begin{pmatrix}
0 & 1\\
1 & 0
\end{pmatrix},
\qquad
\sigma_3=
\begin{pmatrix}
1 & 0\\
0 & -1
\end{pmatrix}.
\]
Fix $m\ge 0$. On each edge we consider the free Dirac expression
\[
D_e=-i\sigma_1\partial_x+m\sigma_3.
\]
At the vertex we impose the Dirac--Kirchhoff conditions
\begin{equation}\label{eq:DK-vertex}
\varphi_1(0)=\cdots=\varphi_N(0),
\qquad
\sum_{e=1}^N \chi_e(0)=0.
\end{equation}

We set
\[
H^1(G;\mathbb{C}^2)=\bigoplus_{e=1}^N H^1(\mathbb{R}_+^{(e)};\mathbb{C}^2),
\]
so that the traces $\psi_e(0)$ are well-defined for $\psi\in H^1(G;\mathbb{C}^2)$.
We define the maximal Dirac operator $D_{\max}$ on $L^2(G;\mathbb{C}^2)$ by
\[
(D_{\max}\psi)_e = D_e\psi_e,
\qquad
\mathrm{Dom}(D_{\max}) = H^1(G;\mathbb{C}^2),
\]
and the Dirac--Kirchhoff realization $D$ as the restriction of $D_{\max}$,
\begin{equation}\label{eq:D-domain}
(D\psi)_e = D_e\psi_e,
\qquad
\mathrm{Dom}(D)=\Bigl\{\psi\in H^1(G;\mathbb{C}^2)\,:\,\psi \text{ satisfies }\eqref{eq:DK-vertex}\Bigr\}.
\end{equation}
As shown in \cite{Borrelli2021}, the operator $D$ is self-adjoint on $L^2(G;\mathbb{C}^2)$.

\subsection{Operator Sobolev spaces and spectral projections}

For $s\in\mathbb{R}$ we define the operator Sobolev space $H_D^s(G)$ as follows.
If $s\ge 0$, set
\begin{equation}\label{eq:HDs}
H_D^s(G)=\mathrm{Dom}\bigl(\langle D\rangle^s\bigr),
\qquad
\|\psi\|_{H_D^s(G)}=\|\langle D\rangle^s\psi\|_{L^2(G)},
\qquad
\langle D\rangle=(I+D^2)^{1/2}.
\end{equation}
If $s<0$, note that the Borel function $\lambda\mapsto \langle \lambda\rangle^{s}$ is bounded on $\mathbb{R}$, hence
$\langle D\rangle^{s}$ is a bounded operator on $L^2(G;\mathbb{C}^2)$ by the functional calculus.
We define $H_D^s(G)$ as the completion of $L^2(G;\mathbb{C}^2)$ with respect to the norm
$\|\psi\|_{H_D^s(G)}=\|\langle D\rangle^s\psi\|_{L^2(G)}$.

We denote by $E_D(\cdot)$ the projection-valued spectral measure of the self-adjoint operator $D$.
We set
\[
\Pi_+ = E_D\bigl((0,\infty)\bigr)=\mathbf{1}_{(0,\infty)}(D),
\qquad
\Pi_- = E_D\bigl((-\infty,0)\bigr)=\mathbf{1}_{(-\infty,0)}(D),
\]
and
\[
\Pi_0 = E_D(\{0\})=\mathbf{1}_{\{0\}}(D),
\]
possibly $\Pi_0=0$.
Then
\[
I=\Pi_+ + \Pi_- + \Pi_0,
\qquad
\psi=\psi_+ + \psi_- + \psi_0,
\qquad
\psi_\bullet=\Pi_\bullet\psi \ \ (\bullet\in\{+,-,0\}).
\]
The operators $\Pi_\bullet$ commute with $D$ and with $\langle D\rangle^s$ for all $s\in\mathbb{R}$, hence they extend to bounded
projections on $H_D^s(G)$.

\begin{lemma}\label{lem:proj-bdd}
For every $s\in\mathbb{R}$ and every $\bullet\in\{+,-,0\}$, the spectral projections $\Pi_\bullet$ satisfy
\[
\|\Pi_\bullet \psi\|_{H_D^s(G)}\le \|\psi\|_{H_D^s(G)}
\qquad\text{for all }\psi\in H_D^s(G).
\]
Moreover,
\[
H_D^s(G)=\Pi_+H_D^s(G)\oplus \Pi_-H_D^s(G)\oplus \Pi_0H_D^s(G),
\]
and the ranges are pairwise orthogonal in $L^2(G;\mathbb{C}^2)$.
\end{lemma}

\begin{proof}
Since $D$ is self-adjoint, each $\Pi_\bullet=\mathbf{1}_{I_\bullet}(D)$ with
$I_+=(0,\infty)$, $I_- = (-\infty,0)$, $I_0=\{0\}$ is an orthogonal projection on $L^2(G;\mathbb{C}^2)$.
In particular, $\|\Pi_\bullet f\|_{L^2(G)}\le \|f\|_{L^2(G)}$ for all $f\in L^2(G;\mathbb{C}^2)$.

By the functional calculus, $\Pi_\bullet$ commutes with any bounded Borel function of $D$.
In particular, for every $s\in\mathbb{R}$ one has
\[
\Pi_\bullet \langle D\rangle^s=\langle D\rangle^s\Pi_\bullet.
\]
Therefore, for any $\psi\in H_D^s(G)$,
\[
\|\Pi_\bullet\psi\|_{H_D^s(G)}
=\|\langle D\rangle^s\Pi_\bullet\psi\|_{L^2(G)}
=\|\Pi_\bullet\langle D\rangle^s\psi\|_{L^2(G)}
\le \|\langle D\rangle^s\psi\|_{L^2(G)}
=\|\psi\|_{H_D^s(G)},
\]
which proves the norm inequality.

For the decomposition, note that the sets $(0,\infty)$, $(-\infty,0)$ and $\{0\}$ are pairwise disjoint and their union is $\mathbb{R}$.
Hence the spectral theorem yields
\[
\Pi_+\Pi_-=\Pi_+\Pi_0=\Pi_-\Pi_0=0,
\qquad
\Pi_+ + \Pi_- + \Pi_0 = I
\]
on $L^2(G;\mathbb{C}^2)$.
Thus
\[
L^2(G;\mathbb{C}^2)=\Pi_+L^2(G;\mathbb{C}^2)\oplus \Pi_-L^2(G;\mathbb{C}^2)\oplus \Pi_0L^2(G;\mathbb{C}^2),
\]
and the sum is orthogonal because the $\Pi_\bullet$ are orthogonal projections with pairwise orthogonal ranges.

Finally, since $\Pi_\bullet$ commutes with $\langle D\rangle^s$, the same direct sum decomposition holds in $H_D^s(G)$:
for any $\psi\in H_D^s(G)$,
\[
\psi = (\Pi_+\psi)+(\Pi_-\psi)+(\Pi_0\psi),
\qquad
\Pi_\bullet\psi\in H_D^s(G),
\]
and the ranges are pairwise $L^2$-orthogonal. This proves the stated decomposition in $H_D^s(G)$.
\end{proof}

\subsection{Spectral transform and Bourgain spaces}

We write
\[
\langle \xi\rangle=(1+|\xi|^{2})^{1/2}.
\]
By the spectral theorem for self-adjoint operators, there exist a Borel measure $\mu_D$ on $\mathbb{R}$,
a measurable family of Hilbert spaces $\{\mathcal{H}_\lambda\}_{\lambda\in\mathbb{R}}$,
and a unitary map
\[
\mathcal{F}_D:L^2(G;\mathbb{C}^2)\to L^2(\mathbb{R},d\mu_D;\mathcal{H}_\lambda)
\]
such that for every bounded Borel function $m$ and every $f\in L^2(G;\mathbb{C}^2)$ one has
\[
\bigl(\mathcal{F}_D(m(D)f)\bigr)(\lambda)=m(\lambda)\,(\mathcal{F}_D f)(\lambda)
\quad\text{for }\mu_D\text{-a.e. }\lambda\in\mathbb{R}.
\]
In particular, for every $f\in \mathrm{Dom}(D)$,
\[
(\mathcal{F}_D(D f))(\lambda)=\lambda(\mathcal{F}_D f)(\lambda)
\quad\text{for }\mu_D\text{-a.e. }\lambda\in\mathbb{R}.
\]
More generally, for every $s\in\mathbb{R}$ and every $f\in \mathrm{Dom}(\langle D\rangle^{s})$ one has
\[
\bigl(\mathcal{F}_D(\langle D\rangle^{s} f)\bigr)(\lambda)
=\langle\lambda\rangle^{s}(\mathcal{F}_D f)(\lambda)
\quad\text{for }\mu_D\text{-a.e. }\lambda\in\mathbb{R}.
\]
By the unitarity of $\mathcal{F}_D$, this yields
\begin{equation}\label{eq:plancherel-HD}
\|f\|_{H_D^s(G)}^2
=\|\langle D\rangle^{s} f\|_{L^2(G)}^{2}
=\int_{\mathbb{R}}\langle\lambda\rangle^{2s}\|(\mathcal{F}_D f)(\lambda)\|_{\mathcal{H}_\lambda}^2\,d\mu_D(\lambda),
\end{equation}
for all $f\in \mathrm{Dom}(\langle D\rangle^{s})$. By density, \eqref{eq:plancherel-HD} extends to all $f\in H_D^s(G)$.

For a spacetime function $u:\mathbb{R}\times G\to\mathbb{C}^2$ we define its joint time--spectral transform by
\[
\widetilde u(\tau,\lambda)=(\mathcal{F}_t\mathcal{F}_D u)(\tau,\lambda),
\qquad
\mathcal{F}_t u(\tau)=\frac{1}{\sqrt{2\pi}}\int_{\mathbb{R}}e^{-it\tau}u(t)\,dt,
\]
where $\mathcal{F}_D$ acts in the space variable for each fixed $t$.

We fix the dense core
\[
\mathcal{S}
=\bigcup_{M\in\mathbb{N}} C_c^\infty\bigl(\mathbb{R};\mathrm{Dom}(\langle D\rangle^M)\bigr),
\]
on which all the transforms and multipliers below are well-defined.

\begin{definition}\label{def:Xsb}
Let $s,b\in\mathbb{R}$.
We define $X^{s,b}(G)$ as the completion of $\mathcal{S}$ under the norm
\begin{equation}\label{eq:Xsb-norm}
\|u\|_{X^{s,b}}^2
=\int_{\mathbb{R}}\int_{\mathbb{R}}
\langle\lambda\rangle^{2s}\langle\tau+\lambda\rangle^{2b}
\|\widetilde u(\tau,\lambda)\|_{\mathcal{H}_\lambda}^2
\,d\tau\,d\mu_D(\lambda).
\end{equation}
For $\bullet\in\{+,-,0\}$ we set
\[
X_\bullet^{s,b}
=\{u\in X^{s,b}(G):\ \Pi_\bullet u=u\},
\qquad
\|u\|_{X_\bullet^{s,b}}=\|u\|_{X^{s,b}},
\]
where $\Pi_\bullet$ acts on $u(t,\cdot)$ for each fixed $t$.
\end{definition}

\begin{lemma}\label{lem:Xsb-shift}
Let $u\in\mathcal{S}$ and set $v(t)=e^{itD}u(t)$. Then
\[
\|u\|_{X^{s,b}}^{2}
=\int_{\mathbb{R}}\int_{\mathbb{R}}
\langle\lambda\rangle^{2s}\langle\tau\rangle^{2b}
\|(\mathcal{F}_t\mathcal{F}_D v)(\tau,\lambda)\|_{\mathcal{H}_\lambda}^{2}
\,d\tau\,d\mu_D(\lambda).
\]
\end{lemma}

\begin{proof}
Let $u\in\mathcal{S}$ and define $v(t)=e^{itD}u(t)$.
Using the spectral representation of $D$, for $\mu_D$-a.e. $\lambda$ we have
\[
(\mathcal{F}_D v)(t,\lambda)=e^{it\lambda}(\mathcal{F}_D u)(t,\lambda).
\]
Taking $\mathcal{F}_t$ yields
\[
(\mathcal{F}_t\mathcal{F}_D v)(\tau,\lambda)
=(\mathcal{F}_t\mathcal{F}_D u)(\tau-\lambda,\lambda)
=\widetilde u(\tau-\lambda,\lambda).
\]
Therefore
\[
\int_{\mathbb{R}}\langle\tau\rangle^{2b}
\|(\mathcal{F}_t\mathcal{F}_D v)(\tau,\lambda)\|_{\mathcal{H}_\lambda}^{2}\,d\tau
=
\int_{\mathbb{R}}\langle\tau\rangle^{2b}
\|\widetilde u(\tau-\lambda,\lambda)\|_{\mathcal{H}_\lambda}^{2}\,d\tau.
\]
Substituting $\sigma=\tau-\lambda$ (so that $\tau=\sigma+\lambda$) gives
\[
\int_{\mathbb{R}}\langle\tau\rangle^{2b}
\|\widetilde u(\tau-\lambda,\lambda)\|_{\mathcal{H}_\lambda}^{2}\,d\tau
=
\int_{\mathbb{R}}\langle\sigma+\lambda\rangle^{2b}
\|\widetilde u(\sigma,\lambda)\|_{\mathcal{H}_\lambda}^{2}\,d\sigma.
\]
Multiplying by $\langle\lambda\rangle^{2s}$, integrating in $\lambda$ against $d\mu_D(\lambda)$,
and using \eqref{eq:Xsb-norm} yields the claimed identity.
\end{proof}

\begin{definition}\label{def:restriction}
Let $T>0$.
The restriction space $X_T^{s,b}$ consists of all functions $u$ defined on $[0,T]$ such that
\[
\|u\|_{X_T^{s,b}}
=\inf\{\|U\|_{X^{s,b}}:\ U\in X^{s,b},\ U(t)=u(t)\ \text{for }t\in[0,T]\}
<\infty.
\]
Similarly, $X_{\bullet,T}^{s,b}$ is defined using $X_\bullet^{s,b}$ for $\bullet\in\{+,-,0\}$.
\end{definition}

\begin{lemma}\label{lem:restriction-banach}
For every $s,b\in\mathbb{R}$ and $T>0$, the space $X_T^{s,b}$ is a Banach space.
Moreover, if $0<T_1\le T_2$ and $u$ is defined on $[0,T_2]$, then
\[
\|u|_{[0,T_1]}\|_{X_{T_1}^{s,b}}\le \|u\|_{X_{T_2}^{s,b}}.
\]
\end{lemma}

\begin{proof}
Define an equivalence relation on $X^{s,b}$ by declaring $U\sim V$ if $U(t)=V(t)$ for $t\in[0,T]$.
Let $\mathcal{N}_T$ be the corresponding subspace
\[
\mathcal{N}_T=\{W\in X^{s,b}:\ W(t)=0 \text{ for } t\in[0,T]\},
\]
where the equality on $[0,T]$ is understood in the sense of restrictions, i.e. $W$ represents the zero element in $X_T^{s,b}$.
Then $X_T^{s,b}$ is naturally identified with the quotient space $X^{s,b}/\mathcal{N}_T$ equipped with the quotient norm
\[
\|u\|_{X_T^{s,b}}=\inf\{\|U\|_{X^{s,b}}:\ U\in X^{s,b},\ U(t)=u(t)\ \text{for }t\in[0,T]\}.
\]

We claim that $\mathcal{N}_T$ is closed in $X^{s,b}$.
Indeed, let $W_n\in\mathcal{N}_T$ and $W_n\to W$ in $X^{s,b}$.
Since $W-W_n$ is an extension of $W|_{[0,T]}$ for every $n$, by the definition of the quotient norm we have
\[
\|W|_{[0,T]}\|_{X_T^{s,b}}\le \|W-W_n\|_{X^{s,b}}.
\]
Letting $n\to\infty$ gives $\|W|_{[0,T]}\|_{X_T^{s,b}}=0$, hence $W\in\mathcal{N}_T$.
Therefore $\mathcal{N}_T$ is closed.

Since $X^{s,b}$ is Banach and $\mathcal{N}_T$ is a closed subspace, the quotient $X^{s,b}/\mathcal{N}_T$
is Banach. Hence $X_T^{s,b}$ is a Banach space.

Finally, let $0<T_1\le T_2$ and let $u$ be defined on $[0,T_2]$.
For every extension $U\in X^{s,b}$ with $U=u$ on $[0,T_2]$, the same $U$ satisfies $U=u|_{[0,T_1]}$ on $[0,T_1]$.
Taking the infimum over all such $U$ yields
\[
\|u|_{[0,T_1]}\|_{X_{T_1}^{s,b}}\le \|u\|_{X_{T_2}^{s,b}}.
\]
\end{proof}

\begin{lemma}\label{lem:embed}
If $b>\frac12$, then
\[
X_T^{s,b}\hookrightarrow C([0,T];H_D^s(G)).
\]
More precisely, there exists a constant $C_b$ such that for every $u\in X_T^{s,b}$,
\begin{equation}\label{eq:Xsb-embed}
\sup_{t\in[0,T]}\|u(t)\|_{H_D^s(G)}\le C_b\|u\|_{X_T^{s,b}}.
\end{equation}
\end{lemma}

\begin{proof}
Let $u\in X_T^{s,b}$. By definition of the restriction norm, for every $\varepsilon>0$ there exists an extension
$U\in X^{s,b}$ such that $U(t)=u(t)$ on $[0,T]$ and
\[
\|U\|_{X^{s,b}}\le \|u\|_{X_T^{s,b}}+\varepsilon.
\]
Fix $t\in\mathbb{R}$. By Fourier inversion in time applied to $(\mathcal{F}_D U)(t,\lambda)$,
\[
(\mathcal{F}_D U)(t,\lambda)
=\frac1{\sqrt{2\pi}}\int_{\mathbb{R}}e^{it\tau}\widetilde U(\tau,\lambda)\,d\tau
\quad\text{in }\mathcal{H}_\lambda.
\]
Using \eqref{eq:plancherel-HD} and Cauchy--Schwarz in $\tau$ with weight $\langle\tau+\lambda\rangle^{b}$, we obtain
\begin{align*}
\|U(t)\|_{H_D^s(G)}^2
&=\int_{\mathbb{R}}\langle\lambda\rangle^{2s}
\|(\mathcal{F}_D U)(t,\lambda)\|_{\mathcal{H}_\lambda}^2\,d\mu_D(\lambda)\\
&\le \frac1{2\pi}\int_{\mathbb{R}}\langle\lambda\rangle^{2s}
\left(\int_{\mathbb{R}}\langle\tau+\lambda\rangle^{-2b}\,d\tau\right)
\left(\int_{\mathbb{R}}\langle\tau+\lambda\rangle^{2b}\|\widetilde U(\tau,\lambda)\|_{\mathcal{H}_\lambda}^2\,d\tau\right)
d\mu_D(\lambda).
\end{align*}
Since $b>\frac12$,
\[
K_b=\int_{\mathbb{R}}\langle\sigma\rangle^{-2b}\,d\sigma<\infty,
\qquad
\int_{\mathbb{R}}\langle\tau+\lambda\rangle^{-2b}\,d\tau=K_b,
\]
and therefore
\[
\|U(t)\|_{H_D^s(G)}^2
\le \frac{K_b}{2\pi}\|U\|_{X^{s,b}}^2.
\]
Taking the supremum over $t\in[0,T]$ yields
\[
\sup_{t\in[0,T]}\|u(t)\|_{H_D^s(G)}
\le \sqrt{\frac{K_b}{2\pi}}\bigl(\|u\|_{X_T^{s,b}}+\varepsilon\bigr).
\]
Letting $\varepsilon\to0$ gives \eqref{eq:Xsb-embed} with $C_b=\sqrt{\frac{K_b}{2\pi}}$.

To prove continuity, fix an extension $U\in X^{s,b}$ with $U=u$ on $[0,T]$.
For $\mu_D$-a.e.\ $\lambda$, define the $\mathcal{H}_\lambda$-valued function
\[
g_\lambda(t)=e^{it\lambda}(\mathcal{F}_D U)(t,\lambda).
\]
Then
\[
(\mathcal{F}_t g_\lambda)(\tau)=\widetilde U(\tau-\lambda,\lambda),
\qquad
\int_{\mathbb{R}}\langle\tau\rangle^{2b}\|(\mathcal{F}_t g_\lambda)(\tau)\|_{\mathcal{H}_\lambda}^2\,d\tau
=\int_{\mathbb{R}}\langle\tau+\lambda\rangle^{2b}\|\widetilde U(\tau,\lambda)\|_{\mathcal{H}_\lambda}^2\,d\tau<\infty.
\]
Hence $g_\lambda\in H^b(\mathbb{R};\mathcal{H}_\lambda)$, and since $b>\frac12$ we have
$H^b(\mathbb{R};\mathcal{H}_\lambda)\hookrightarrow C(\mathbb{R};\mathcal{H}_\lambda)$.
Therefore $t\mapsto (\mathcal{F}_D U)(t,\lambda)=e^{-it\lambda}g_\lambda(t)$ is continuous for $\mu_D$-a.e.\ $\lambda$.

Let $t_n\to t$ with $t_n,t\in[0,T]$. By \eqref{eq:plancherel-HD},
\[
\|U(t_n)-U(t)\|_{H_D^s(G)}^2
=\int_{\mathbb{R}}\langle\lambda\rangle^{2s}
\|(\mathcal{F}_D U)(t_n,\lambda)-(\mathcal{F}_D U)(t,\lambda)\|_{\mathcal{H}_\lambda}^2\,d\mu_D(\lambda).
\]
The integrand converges to $0$ for $\mu_D$-a.e.\ $\lambda$, and by the first part of the proof,
\[
\|(\mathcal{F}_D U)(t_n,\lambda)-(\mathcal{F}_D U)(t,\lambda)\|_{\mathcal{H}_\lambda}
\le 2\sup_{r\in[0,T]}\|(\mathcal{F}_D U)(r,\lambda)\|_{\mathcal{H}_\lambda}
\le \frac{2\sqrt{K_b}}{\sqrt{2\pi}}
\left(\int_{\mathbb{R}}\langle\tau+\lambda\rangle^{2b}\|\widetilde U(\tau,\lambda)\|_{\mathcal{H}_\lambda}^2\,d\tau\right)^{1/2}.
\]
The square of the right-hand side is integrable in $\lambda$ with weight $\langle\lambda\rangle^{2s}$ because $U\in X^{s,b}$.
Thus dominated convergence yields $\|U(t_n)-U(t)\|_{H_D^s(G)}\to0$.
Since $u=U$ on $[0,T]$, we conclude $u\in C([0,T];H_D^s(G))$.
\end{proof}

\section{Linear and Bilinear Estimates}\label{sec:estimates}

Throughout this section, $X^{s,b}$ and $X_T^{s,b}$ are the Bourgain spaces associated with $D$
as in Definitions~\ref{def:Xsb}--\ref{def:restriction}.
We fix a nonnegative cutoff $\eta\in C_c^\infty(\mathbb{R})$ such that
\begin{equation}\label{eq:cutoff-basic}
\eta(t)=1\ \text{for }|t|\le 1,
\qquad
\mathrm{supp}\,\eta\subset\{t:|t|\le 2\}.
\end{equation}
For $T\in(0,1]$ we set $\eta_T(t)=\eta(t/T)$.

\begin{lemma}\label{lem:free}
Let $\eta\in C_c^\infty(\mathbb{R})$.
For any $s\in\mathbb{R}$ and any $b\in\mathbb{R}$,
\begin{equation}\label{eq:free-cutoff}
\|\eta(t)e^{-itD}\psi_0\|_{X^{s,b}}\le C_{\eta,b}\|\psi_0\|_{H_D^s(G)}.
\end{equation}
Consequently, if $\eta$ satisfies \eqref{eq:cutoff-basic}, then for any $T\in(0,1]$,
\begin{equation}\label{eq:free-restriction}
\|e^{-itD}\psi_0\|_{X_T^{s,b}}\le C_{\eta,b}\|\psi_0\|_{H_D^s(G)}.
\end{equation}
\end{lemma}

\begin{proof}
Let $\psi_0\in H_D^s(G)$ and set $U(t)=\eta(t)e^{-itD}\psi_0$.
Applying the spectral transform $\mathcal{F}_D$ in the space variable gives
\[
(\mathcal{F}_D U)(t,\lambda)=\eta(t)e^{-it\lambda}(\mathcal{F}_D\psi_0)(\lambda)
\quad\text{in }\mathcal{H}_\lambda.
\]
Taking the Fourier transform in time yields
\[
\widetilde U(\tau,\lambda)
=(\mathcal{F}_t\mathcal{F}_D U)(\tau,\lambda)
=\widehat{\eta}(\tau+\lambda)\,(\mathcal{F}_D\psi_0)(\lambda),
\]
where $\widehat{\eta}=\mathcal{F}_t\eta$.
Hence, by the definition of the $X^{s,b}$ norm,
\begin{align*}
\|U\|_{X^{s,b}}^2
&=\int_{\mathbb{R}}\int_{\mathbb{R}}
\langle\lambda\rangle^{2s}\langle\tau+\lambda\rangle^{2b}
|\widehat{\eta}(\tau+\lambda)|^2
\|(\mathcal{F}_D\psi_0)(\lambda)\|_{\mathcal{H}_\lambda}^2
\,d\tau\,d\mu_D(\lambda)\\
&=\left(\int_{\mathbb{R}}\langle\sigma\rangle^{2b}|\widehat{\eta}(\sigma)|^2\,d\sigma\right)
\int_{\mathbb{R}}\langle\lambda\rangle^{2s}\|(\mathcal{F}_D\psi_0)(\lambda)\|_{\mathcal{H}_\lambda}^2\,d\mu_D(\lambda),
\end{align*}
where we used the change of variables $\sigma=\tau+\lambda$.
By \eqref{eq:plancherel-HD}, the second integral equals $\|\psi_0\|_{H_D^s(G)}^2$.
This proves \eqref{eq:free-cutoff} with
\[
C_{\eta,b}^2=\int_{\mathbb{R}}\langle\sigma\rangle^{2b}|\widehat{\eta}(\sigma)|^2\,d\sigma,
\]
which is finite since $\widehat{\eta}$ is a Schwartz function.

To prove \eqref{eq:free-restriction}, assume that $\eta$ satisfies \eqref{eq:cutoff-basic}.
For $T\in(0,1]$, the function $U(t)=\eta(t)e^{-itD}\psi_0$ coincides with $e^{-itD}\psi_0$ on $[0,T]$.
Hence Definition~\ref{def:restriction} gives
\[
\|e^{-itD}\psi_0\|_{X_T^{s,b}}
\le \|U\|_{X^{s,b}}
\le C_{\eta,b}\|\psi_0\|_{H_D^s(G)}.
\]
\end{proof}

Let $\mathcal H$ be a Hilbert space. For $\lambda\in\mathbb{R}$ and $\theta\in\mathbb{R}$, we define the shifted Sobolev space
$H_\lambda^\theta(\mathbb{R};\mathcal H)$ by
\[
\|g\|_{H_\lambda^\theta(\mathbb{R};\mathcal H)}^2
=\int_{\mathbb{R}}\langle\tau+\lambda\rangle^{2\theta}\|\widehat g(\tau)\|_{\mathcal H}^2\,d\tau,
\qquad
\widehat g=\mathcal{F}_t g.
\]

\begin{lemma}\label{lem:scalar-duhamel}
Let $b\in\left(\frac12,1\right)$ and $b'\in\left(0,\frac12\right)$.
Let $\eta$ satisfy \eqref{eq:cutoff-basic} and let $T\in(0,1]$.
For any $\lambda\in\mathbb{R}$ and any Hilbert space $\mathcal H$, define
\[
\mathcal T_\lambda f(t)
=\eta_T(t)\int_0^t e^{-i(t-\tau)\lambda}f(\tau)\,d\tau,
\qquad
\eta_T(t)=\eta(t/T),
\]
for $f\in H_\lambda^{b'-1}(\mathbb{R};\mathcal H)$.
Then there exists $C$ depending only on $b,b'$ and $\eta$ such that
\begin{equation}\label{eq:scalar-duhamel}
\|\mathcal T_\lambda f\|_{H_\lambda^{b}(\mathbb{R};\mathcal H)}
\le C\,T^{\,1-b+b'}\|f\|_{H_\lambda^{b'-1}(\mathbb{R};\mathcal H)}.
\end{equation}
The constant $C$ is independent of $\lambda$ and $T$.
\end{lemma}

\begin{proof}
Let $g(t)=e^{it\lambda}\mathcal T_\lambda f(t)$ and $F(t)=e^{it\lambda}f(t)$. Then
\[
g(t)=\eta_T(t)\int_0^t F(\tau)\,d\tau.
\]
By definition of the shifted Sobolev norms,
\[
\|\mathcal T_\lambda f\|_{H_\lambda^{b}(\mathbb{R};\mathcal H)}=\|g\|_{H^{b}(\mathbb{R};\mathcal H)},
\qquad
\|f\|_{H_\lambda^{b'-1}(\mathbb{R};\mathcal H)}=\|F\|_{H^{b'-1}(\mathbb{R};\mathcal H)}.
\]
Hence it suffices to prove the estimate in the case $\lambda=0$, namely
\begin{equation}\label{eq:scalar-duhamel-lam0}
\Bigl\|\eta_T(t)\int_0^t F(\tau)\,d\tau\Bigr\|_{H^{b}(\mathbb{R};\mathcal H)}
\le C\,T^{\,1-b+b'}\|F\|_{H^{b'-1}(\mathbb{R};\mathcal H)}.
\end{equation}

Define the time-localized Duhamel operator
\[
(\mathcal S_T F)(t)=\eta_T(t)\int_0^t F(\tau)\,d\tau.
\]
Since $\mathcal H$ is a Hilbert space, the estimate \eqref{eq:scalar-duhamel-lam0} follows from the scalar case
by applying the bound componentwise with respect to an orthonormal basis and using Plancherel.
The scalar bound is a standard time cutoff estimate in Bourgain-type spaces; see, for instance,
\cite[Lemma~2.1]{TaoBilinearDispersive}.
Therefore,
\[
\|\mathcal S_T F\|_{H^{b}(\mathbb{R};\mathcal H)}
\le C\,T^{\,1-b+b'}\|F\|_{H^{b'-1}(\mathbb{R};\mathcal H)},
\]
with $C$ depending only on $b,b'$ and $\eta$, and independent of $T\in(0,1]$.
This proves \eqref{eq:scalar-duhamel-lam0}, hence \eqref{eq:scalar-duhamel}.
\end{proof}

\begin{lemma}\label{lem:duhamel}
Let $b\in\left(\frac12,1\right)$ and let $b'\in\left(0,\frac12\right)$.
For $T\in(0,1]$ and any $F\in X_T^{s,b'-1}$,
\begin{equation}\label{eq:duhamel}
\left\|\int_0^t e^{-i(t-\tau)D}F(\tau)\,d\tau\right\|_{X_T^{s,b}}
\le C\,T^{\,1-b+b'}\|F\|_{X_T^{s,b'-1}},
\end{equation}
where $C$ depends only on $b,b'$ and $\eta$ in \eqref{eq:cutoff-basic}.
\end{lemma}

\begin{proof}
Let $F\in X_T^{s,b'-1}$ and let $F^{\mathrm{ext}}\in X^{s,b'-1}$ be any extension such that
$F^{\mathrm{ext}}(t)=F(t)$ for $t\in[0,T]$.
Set
\[
\Psi(t)=\eta_T(t)\int_0^t e^{-i(t-\tau)D}F^{\mathrm{ext}}(\tau)\,d\tau.
\]
Since $\eta_T(t)=1$ for $t\in[0,T]$, we have
\[
\Psi(t)=\int_0^t e^{-i(t-\tau)D}F(\tau)\,d\tau
\qquad\text{for all }t\in[0,T].
\]
Hence, by Definition~\ref{def:restriction},
\[
\left\|\int_0^t e^{-i(t-\tau)D}F(\tau)\,d\tau\right\|_{X_T^{s,b}}
\le \|\Psi\|_{X^{s,b}}.
\]

Apply $\mathcal{F}_D$ in the space variable. For $\mu_D$-a.e.\ $\lambda$ set
\[
F_\lambda(t)=(\mathcal{F}_D F^{\mathrm{ext}})(t,\lambda)\in\mathcal{H}_\lambda,
\qquad
\Psi_\lambda(t)=(\mathcal{F}_D\Psi)(t,\lambda)
=\eta_T(t)\int_0^t e^{-i(t-\tau)\lambda}F_\lambda(\tau)\,d\tau.
\]
By Lemma~\ref{lem:scalar-duhamel} with $\mathcal H=\mathcal{H}_\lambda$,
\[
\|\Psi_\lambda\|_{H_\lambda^b(\mathbb{R};\mathcal{H}_\lambda)}
\le C\,T^{\,1-b+b'}\|F_\lambda\|_{H_\lambda^{b'-1}(\mathbb{R};\mathcal{H}_\lambda)}.
\]
Multiplying by $\langle\lambda\rangle^{s}$, squaring, and integrating in $\lambda$ with respect to $\mu_D$,
the definition \eqref{eq:Xsb-norm} yields
\[
\|\Psi\|_{X^{s,b}}
\le C\,T^{\,1-b+b'}\|F^{\mathrm{ext}}\|_{X^{s,b'-1}}.
\]
Combining the previous inequalities gives
\[
\left\|\int_0^t e^{-i(t-\tau)D}F(\tau)\,d\tau\right\|_{X_T^{s,b}}
\le C\,T^{\,1-b+b'}\|F^{\mathrm{ext}}\|_{X^{s,b'-1}}.
\]
Taking the infimum over all such extensions $F^{\mathrm{ext}}$ and using Definition~\ref{def:restriction},
we obtain \eqref{eq:duhamel}.
\end{proof}

\begin{lemma}\label{lem:reflection-extension}
For $\kappa\in\{1,-1\}$ define the boundary-adapted reflection operator
\[
(\mathcal E_\kappa z)(x)=
\begin{cases}
\frac{1}{\sqrt2}z(x), & x\ge 0,\\[4pt]
\frac{1}{\sqrt2}\kappa\sigma_3\,z(-x), & x<0.
\end{cases}
\]
Then the following statements hold.
\begin{enumerate}
\item $\mathcal E_\kappa:L^2(\mathbb R_+;\mathbb C^2)\to L^2(\mathbb R;\mathbb C^2)$ is an isometry onto its range.
\item If $z\in\mathrm{Dom}(D_\kappa)$, then $\mathcal E_\kappa z\in H^1(\mathbb R;\mathbb C^2)$ and
\[
D_{\mathbb R}\mathcal E_\kappa z=\mathcal E_\kappa D_\kappa z.
\]
\item For every bounded Borel function $\Phi:\mathbb R\to\mathbb C$,
\[
\Phi(D_{\mathbb R})\,\mathcal E_\kappa=\mathcal E_\kappa\,\Phi(D_\kappa).
\]
In particular,
\[
\|\mathcal E_\kappa f\|_{X^{s,\gamma}(\mathbb R)}=\|f\|_{X^{s,\gamma}(\mathbb R_+)},
\qquad
\|\mathcal E_\kappa f\|_{X_{\pm}^{s,\gamma}(\mathbb R)}=\|f\|_{X_{\pm}^{s,\gamma}(\mathbb R_+)}.
\]
\end{enumerate}
\end{lemma}

\begin{proof}
The $L^2$ isometry is immediate from the definition and the identity $|\sigma_3 w|=|w|$.
If $z\in\mathrm{Dom}(D_\kappa)$, then $(I-\kappa\sigma_3)z(0)=0$, equivalently $z(0)=\kappa\sigma_3 z(0)$, so the two traces defining $\mathcal E_\kappa z$ agree at $x=0$.
Hence $\mathcal E_\kappa z\in H^1(\mathbb R;\mathbb C^2)$.
For $x\neq 0$, a direct computation using $\sigma_3\sigma_1=-\sigma_1\sigma_3$ and $\sigma_3^2=I$ gives
\[
D_{\mathbb R}\mathcal E_\kappa z=\mathcal E_\kappa D_\kappa z.
\]
Since $D_{\mathbb R}$ and $D_\kappa$ are self-adjoint and $\mathcal E_\kappa$ is bounded, the intertwining extends from the resolvents to the bounded Borel functional calculus.
Applying this with $\Phi(\lambda)=\langle\lambda\rangle^s\langle\tau+\lambda\rangle^\gamma$ and integrating in $\tau$ yields the Bourgain-space isometries.
Taking $\Phi$ to be the characteristic functions of $(0,\infty)$ and $(-\infty,0)$ gives the corresponding identities for the spectral projections $\Pi_\pm$.
\end{proof}

\begin{proposition}\label{prop:bilinear-mixed}
Let $s>-\frac18$ and let $b\in\left(\frac12,1\right)$.
Let $b_0$ satisfy
\[
-\frac12<b_0\le 0,
\qquad
b_0-2s\le 0,
\qquad
2b+4s>1.
\]
Then for all $u\in X_{+,T}^{s,b}$ and $v\in X_{-,T}^{s,b}$,
\[
\|uv\|_{X_T^{s,b_0}}
\le C_N\,\|u\|_{X_{+,T}^{s,b}}\|v\|_{X_{-,T}^{s,b}},
\]
where $C_N$ depends on $s,b,b_0$ and $N$, but is independent of $T\in(0,1]$.
Here $uv$ denotes the componentwise product in $\mathbb{C}^2$.
\end{proposition}

\begin{proof}
Step 1. Let $u\in X_{+,T}^{s,b}$ and $v\in X_{-,T}^{s,b}$ and fix $\varepsilon>0$.
By the definition of the restriction norm, there exist extensions
$U\in X_+^{s,b}$ and $V\in X_-^{s,b}$ such that $U=u$ and $V=v$ on $[0,T]$ and
\[
\|U\|_{X_+^{s,b}}\le \|u\|_{X_{+,T}^{s,b}}+\varepsilon,
\qquad
\|V\|_{X_-^{s,b}}\le \|v\|_{X_{-,T}^{s,b}}+\varepsilon.
\]
Since $UV$ is an extension of $uv$ on $[0,T]$, the restriction norm gives
\[
\|uv\|_{X_T^{s,b_0}}\le \|UV\|_{X^{s,b_0}}.
\]
Thus it suffices to prove the global estimate
\begin{equation}\label{eq:goal-global-merged}
\|UV\|_{X^{s,b_0}}
\le C_N\,\|U\|_{X_+^{s,b}}\|V\|_{X_-^{s,b}},
\end{equation}
with $C_N$ independent of $T$ and of the chosen extensions.
Once \eqref{eq:goal-global-merged} is proved, letting $\varepsilon\to 0$ yields the desired estimate.

Step 2. Write the $N$-star graph as $\mathcal G=\bigsqcup_{e=1}^N \mathbb R_+$ and identify
\[
L^2(\mathcal G;\mathbb C^2)=\bigoplus_{e=1}^N L^2(\mathbb R_+;\mathbb C^2),
\qquad
W=(W_1,\dots,W_N).
\]
Let
\[
\sigma_1=\begin{pmatrix}0&1\\1&0\end{pmatrix},
\qquad
\sigma_3=\begin{pmatrix}1&0\\0&-1\end{pmatrix},
\qquad
D=-i\sigma_1\partial_x+m\sigma_3
\]
act edgewise, and assume the Dirac--Kirchhoff vertex condition
\[
(W_e)_1(0)=(W_{e'})_1(0)\ \text{for all } e,e',
\qquad
\sum_{e=1}^N (W_e)_2(0)=0.
\]
Define $f^{(1)}=\frac{1}{\sqrt N}(1,\dots,1)\in\mathbb C^N$ and choose
$f^{(2)},\dots,f^{(N)}$ so that $\{f^{(j)}\}_{j=1}^N$ is an orthonormal basis of $\mathbb C^N$.
Let $U$ be the unitary $N\times N$ matrix whose $j$-th row is $(f^{(j)})^\ast$, and define
a unitary operator $\mathcal U$ acting only on the edge index by
\[
(\mathcal U W)_j=\sum_{e=1}^N U_{je}\,W_e,
\qquad j=1,\dots,N.
\]
Since $\mathcal U$ is independent of $x$, it commutes with $\partial_x$ and with $m\sigma_3$; hence
\[
(\mathcal U D W)_j
=-i\sigma_1\partial_x (\mathcal U W)_j+m\sigma_3(\mathcal U W)_j.
\]
It remains to identify the transformed boundary condition. Write $w=\mathcal U W$ and
$W_e(0)=(\alpha_e,\beta_e)^{\mathsf T}$. The continuity of the first component implies
$\alpha_1=\cdots=\alpha_N=\alpha$, hence $(\alpha_e)_{e=1}^N=\alpha(1,\dots,1)$, so
\[
(w_1)_1(0)=\sum_{e=1}^N U_{1e}\alpha_e=\sqrt N\,\alpha,
\qquad
(w_j)_1(0)=\sum_{e=1}^N U_{je}\alpha_e=0\quad (j=2,\dots,N),
\]
because $f^{(j)}\perp (1,\dots,1)$ for $j\ge 2$. The Kirchhoff condition gives $\sum_{e=1}^N\beta_e=0$ and thus
\[
(w_1)_2(0)=\sum_{e=1}^N U_{1e}\beta_e=\frac{1}{\sqrt N}\sum_{e=1}^N\beta_e=0.
\]
Therefore the transformed boundary conditions are
\[
(w_1)_2(0)=0,
\qquad
(w_j)_1(0)=0\quad (j=2,\dots,N).
\]
For $\kappa\in\{1,-1\}$ define the half-line operator $D_\kappa$ on $L^2(\mathbb R_+;\mathbb C^2)$ by
\[
D_\kappa=-i\sigma_1\partial_x+m\sigma_3,
\qquad
\mathrm{Dom}(D_\kappa)=\bigl\{z\in H^1(\mathbb R_+;\mathbb C^2): (I-\kappa\sigma_3)z(0)=0\bigr\}.
\]
Since $(I-\sigma_3)(a,b)^{\mathsf T}=(0,2b)^{\mathsf T}$ and $(I+\sigma_3)(a,b)^{\mathsf T}=(2a,0)^{\mathsf T}$,
the above boundary conditions mean $w_1\in\mathrm{Dom}(D_{+1})$ and $w_j\in\mathrm{Dom}(D_{-1})$ for $j\ge 2$.
Hence
\[
\mathcal U D\mathcal U^{-1}
= D_{+1}\oplus D_{-1}^{\oplus (N-1)}
\quad\text{on}\quad
\bigoplus_{j=1}^N L^2(\mathbb R_+;\mathbb C^2).
\]
In particular, if $\widehat W=\mathcal U W=(\widehat W_1,\dots,\widehat W_N)$, then the Bourgain norms
associated with $D$ satisfy, for every $\beta\in\mathbb R$,
\begin{equation}\label{eq:X-decomp-merged}
\|W\|_{X^{s,\beta}}^2=\sum_{j=1}^N \|\widehat W_j\|_{X^{s,\beta}(\mathbb R_+)}^2,
\qquad
\|W\|_{X_{\pm}^{s,\beta}}^2=\sum_{j=1}^N \|\widehat W_j\|_{X_{\pm}^{s,\beta}(\mathbb R_+)}^2,
\end{equation}
because functional calculus for a direct sum is componentwise, and $\mathcal U$ is unitary.

Moreover, since $\mathcal U$ acts only on the edge index, there exist constants $c_{j k \ell}$ such that for all spacetime
functions $F,G$,
\begin{equation}\label{eq:mode-coupling-merged}
(\mathcal U(FG))_j=\sum_{k=1}^N\sum_{\ell=1}^N c_{j k \ell}\,(\mathcal U F)_k\,(\mathcal U G)_\ell,
\qquad j=1,\dots,N,
\end{equation}
with
\[
c_{j k \ell}=\sum_{e=1}^N U_{j e}\,\overline{U_{k e}}\,\overline{U_{\ell e}}.
\]
Define the algebraic constant
\begin{equation}\label{eq:CN-alg-merged}
C_N^{\mathrm{alg}}
=\sup_{\mathbf a,\mathbf b\ne 0}
\frac{\left(\sum_{j=1}^N\left|\sum_{k=1}^N\sum_{\ell=1}^N |c_{j k \ell}|\,a_k b_\ell\right|^2\right)^{1/2}}
{\left(\sum_{k=1}^N|a_k|^2\right)^{1/2}\left(\sum_{\ell=1}^N|b_\ell|^2\right)^{1/2}}<\infty.
\end{equation}
This constant is finite since it is the operator norm of a bilinear map between finite-dimensional Hilbert spaces.
A crude explicit bound follows from $|c_{j k \ell}|\le 1$: indeed,
\[
|c_{j k \ell}|
\le \sum_{e=1}^N |U_{j e}|\,|U_{k e}|\,|U_{\ell e}|
\le \sum_{e=1}^N |U_{j e}|\,|U_{k e}|
\le \Bigl(\sum_{e=1}^N|U_{j e}|^2\Bigr)^{1/2}\Bigl(\sum_{e=1}^N|U_{k e}|^2\Bigr)^{1/2}
=1,
\]
hence for $\mathbf a,\mathbf b\in\mathbb R^N$ with nonnegative entries,
\[
\sum_{k=1}^N\sum_{\ell=1}^N |c_{j k \ell}|\,a_k b_\ell
\le \sum_{k=1}^N a_k\sum_{\ell=1}^N b_\ell
\le \sqrt N\,\|\mathbf a\|_{\ell^2}\,\sqrt N\,\|\mathbf b\|_{\ell^2},
\]
and therefore $\|B(\mathbf a,\mathbf b)\|_{\ell^2}\le N^{3/2}\|\mathbf a\|_{\ell^2}\|\mathbf b\|_{\ell^2}$, so
\[
C_N^{\mathrm{alg}}\le N^{3/2}.
\]

Step 3. We now pass from each half-line channel to the whole line by means of the reflection operator $\mathcal E_\kappa$ from Lemma~\ref{lem:reflection-extension}.
This is the Dirac analogue of an even/odd extension, adapted to the boundary condition encoded by $\kappa$.
In particular,
\begin{equation}\label{eq:E-isom-merged}
\|\mathcal E_\kappa f\|_{X^{s,\gamma}(\mathbb R)}=\|f\|_{X^{s,\gamma}(\mathbb R_+)}
\end{equation}
and
\begin{equation}\label{eq:E-isom-signed-merged}
\|\mathcal E_\kappa f\|_{X_{\pm}^{s,\gamma}(\mathbb R)}=\|f\|_{X_{\pm}^{s,\gamma}(\mathbb R_+)}.
\end{equation}

Step 4. Let $\widehat U=\mathcal U U$ and $\widehat V=\mathcal U V$.
Fix $k,\ell\in\{1,\dots,N\}$ and let $\kappa_k,\kappa_\ell\in\{1,-1\}$ be the boundary types of the $k$-th and $\ell$-th channels
as determined in Step 2.
Set
\[
F=\mathcal E_{\kappa_k}\widehat U_k,
\qquad
G=\mathcal E_{\kappa_\ell}\widehat V_\ell.
\]
Then on $\mathbb R_+$ one has $F=\frac{1}{\sqrt2}\widehat U_k$ and $G=\frac{1}{\sqrt2}\widehat V_\ell$, hence
\[
(FG)|_{\mathbb R_+}=\frac12\,\widehat U_k\widehat V_\ell.
\]
By the definition of the restriction norm on $\mathbb R_+$, the function $FG$ is an admissible extension of
$\frac12\,\widehat U_k\widehat V_\ell$, so
\begin{equation}\label{eq:halfline-restrict-merged}
\|\widehat U_k\widehat V_\ell\|_{X^{s,b_0}(\mathbb R_+)}
\le 2\,\|FG\|_{X^{s,b_0}(\mathbb R)}.
\end{equation}
By \eqref{eq:E-isom-signed-merged}, we have $F\in X_+^{s,b}(\mathbb R)$ and $G\in X_-^{s,b}(\mathbb R)$ with
\[
\|F\|_{X_+^{s,b}(\mathbb R)}=\|\widehat U_k\|_{X_+^{s,b}(\mathbb R_+)},
\qquad
\|G\|_{X_-^{s,b}(\mathbb R)}=\|\widehat V_\ell\|_{X_-^{s,b}(\mathbb R_+)}.
\]
Under the assumptions on $s,b,b_0$, Machihara's mixed-sign bilinear estimate \cite[Proposition 1]{Machihara2005}
yields a constant $C=C(s,b,b_0)$ such that for all
$F\in X_+^{s,b}(\mathbb R)$ and $G\in X_-^{s,b}(\mathbb R)$,
\[
\|FG\|_{X^{s,b_0}(\mathbb R)}
\le C\,\|F\|_{X_+^{s,b}(\mathbb R)}\,\|G\|_{X_-^{s,b}(\mathbb R)}.
\]
Combining with \eqref{eq:halfline-restrict-merged} gives the pairwise half-line bound
\begin{equation}\label{eq:pairwise-merged}
\|\widehat U_k \widehat V_\ell\|_{X^{s,b_0}(\mathbb R_+)}
\le 2C\,\|\widehat U_k\|_{X_+^{s,b}(\mathbb R_+)}\,\|\widehat V_\ell\|_{X_-^{s,b}(\mathbb R_+)}.
\end{equation}

Step 5. By \eqref{eq:mode-coupling-merged} and the triangle inequality,
\[
\|(\mathcal U(UV))_j\|_{X^{s,b_0}(\mathbb R_+)}
\le \sum_{k=1}^N\sum_{\ell=1}^N |c_{j k \ell}|\,
\|\widehat U_k \widehat V_\ell\|_{X^{s,b_0}(\mathbb R_+)}.
\]
Set $a_k=\|\widehat U_k\|_{X_+^{s,b}(\mathbb R_+)}$ and $b_\ell=\|\widehat V_\ell\|_{X_-^{s,b}(\mathbb R_+)}$.
Using \eqref{eq:pairwise-merged},
\[
\|(\mathcal U(UV))_j\|_{X^{s,b_0}(\mathbb R_+)}
\le 2C\,\sum_{k=1}^N\sum_{\ell=1}^N |c_{j k \ell}|\,a_k b_\ell.
\]
Taking the $\ell^2$ norm in $j$ and using \eqref{eq:CN-alg-merged} yields
\[
\left(\sum_{j=1}^N\|(\mathcal U(UV))_j\|_{X^{s,b_0}(\mathbb R_+)}^2\right)^{1/2}
\le 2C\,C_N^{\mathrm{alg}}
\left(\sum_{k=1}^N a_k^2\right)^{1/2}\left(\sum_{\ell=1}^N b_\ell^2\right)^{1/2}.
\]
By \eqref{eq:X-decomp-merged} applied to $W=UV$, $W=U$ and $W=V$,
\[
\|UV\|_{X^{s,b_0}}
\le 2C\,C_N^{\mathrm{alg}}\,\|U\|_{X_+^{s,b}}\|V\|_{X_-^{s,b}}.
\]
This is \eqref{eq:goal-global-merged} with $C_N=2C\,C_N^{\mathrm{alg}}$.
Combining with Step 1 and letting $\varepsilon\to 0$ yields the claimed estimate.
\end{proof}

\begin{lemma}\label{lem:nonlinear}
Let $s>-\frac18$ and let $b\in\left(\frac12,1\right)$.
Assume that $b_0$ satisfies
\[
-\frac12<b_0\le 0,
\qquad
b_0-2s\le 0,
\qquad
2b+4s>1.
\]
Fix $b'\in\left(0,\frac12\right)$ with $b'<b$.
Let $\mathcal N(\psi)=\mathcal B(\Pi_+\psi,\Pi_-\psi)$, where $\mathcal B$ is bilinear and
\[
|\mathcal B(a,b)|\le C_{\mathcal B}|a|\,|b|,
\qquad a,b\in\mathbb C^2.
\]
Then for every $T\in(0,1]$ and every $\psi\in X_T^{s,b}$,
\begin{equation}\label{eq:quad-nlin}
\|\mathcal N(\psi)\|_{X_T^{s,b'-1}}
\le C_{\mathcal B}\,C_N\,\|\Pi_+\psi\|_{X_{+,T}^{s,b}}\;\|\Pi_-\psi\|_{X_{-,T}^{s,b}}
\le C_{\mathcal B}\,C_N\,\|\psi\|_{X_T^{s,b}}^{2}.
\end{equation}
Moreover, for $\psi_1,\psi_2\in X_T^{s,b}$,
\begin{equation}\label{eq:quad-lip}
\|\mathcal N(\psi_1)-\mathcal N(\psi_2)\|_{X_T^{s,b'-1}}
\le C_{\mathcal B}\,C_N\bigl(\|\psi_1\|_{X_T^{s,b}}+\|\psi_2\|_{X_T^{s,b}}\bigr)\|\psi_1-\psi_2\|_{X_T^{s,b}}.
\end{equation}
\end{lemma}

\begin{proof}
Since $b'-1<b_0$, one has $\|F\|_{X_T^{s,b'-1}}\le \|F\|_{X_T^{s,b_0}}$.

By $|\mathcal B(a,b)|\le C_{\mathcal B}|a||b|$, each component of $\mathcal N(\psi)$ is a linear combination of products of components of $\Pi_+\psi$ and $\Pi_-\psi$ with coefficients bounded by $C_{\mathcal B}$.
Applying Proposition~\ref{prop:bilinear-mixed} to these products yields
\[
\|\mathcal N(\psi)\|_{X_T^{s,b_0}}
\le C_{\mathcal B}\,C_N\,\|\Pi_+\psi\|_{X_{+,T}^{s,b}}\;\|\Pi_-\psi\|_{X_{-,T}^{s,b}}.
\]
This gives the first inequality in \eqref{eq:quad-nlin}, and the second follows from
$\|\Pi_\pm\psi\|_{X_T^{s,b}}\le \|\psi\|_{X_T^{s,b}}$.

For $\psi_1,\psi_2\in X_T^{s,b}$, bilinearity gives
\[
\mathcal N(\psi_1)-\mathcal N(\psi_2)
=\mathcal B\bigl(\Pi_+(\psi_1-\psi_2),\Pi_-\psi_1\bigr)
+\mathcal B\bigl(\Pi_+\psi_2,\Pi_-(\psi_1-\psi_2)\bigr).
\]
Applying the previous estimate to each term and using $\|\Pi_\pm(\psi_1-\psi_2)\|_{X_T^{s,b}}\le \|\psi_1-\psi_2\|_{X_T^{s,b}}$
yields \eqref{eq:quad-lip}.
\end{proof}

\section{Proof of Local Well-Posedness}

\begin{lemma}\label{lem:mild-weak}
Let $s>-\frac18$ and let $b>\frac12$.
Assume that $\psi\in X_T^{s,b}$ satisfies, for every $t\in[0,T]$,
\[
\psi(t)=e^{-itD}\psi_0+i\int_0^t e^{-i(t-\tau)D}\,\mathcal N(\psi(\tau))\,d\tau
\quad\text{in }H_D^s(G).
\]
Then $\psi$ is a distributional solution of \eqref{eq:NLDE} on $(0,T)$ with values in $H_D^{s-1}(G)$, namely
\[
i\partial_t\psi=D\psi-\mathcal N(\psi)
\quad\text{in }\mathcal D'\bigl((0,T);H_D^{s-1}(G)\bigr).
\]
\end{lemma}

\begin{proof}
Since $b>\frac12$, Lemma~\ref{lem:embed} yields $\psi\in C([0,T];H_D^s(G))$.
Fix $b'\in(0,\frac12)$ with $b'<b$ and choose $b_0$ as in Proposition~\ref{prop:bilinear-mixed}.
By Lemma~\ref{lem:nonlinear} we have $F=\mathcal N(\psi)\in X_T^{s,b'-1}$.

Let $\zeta\in C_c^\infty(0,T)$ and $\phi\in\mathrm{Dom}(D)$.
We show that
\begin{equation}\label{eq:weak-form-mild}
\int_0^T \bigl\langle \psi(t),-i\zeta'(t)\phi-\zeta(t)D\phi\bigr\rangle_{L^2(G)}\,dt
+\int_0^T \langle F(t),\zeta(t)\phi\rangle_{L^2(G)}\,dt=0.
\end{equation}
We split $\psi=e^{-itD}\psi_0+iW$, where
\[
W(t)=\int_0^t e^{-i(t-\tau)D}F(\tau)\,d\tau.
\]

For $\psi_0\in L^2(G;\mathbb C^2)$ one checks that
$t\mapsto \langle e^{-itD}\psi_0,\phi\rangle_{L^2(G)}$ is $C^1$ and
\[
\frac{d}{dt}\langle e^{-itD}\psi_0,\phi\rangle_{L^2}
=\langle e^{-itD}\psi_0,-iD\phi\rangle_{L^2}.
\]
Since $\zeta\in C_c^\infty(0,T)$, integration by parts yields
\begin{equation}\label{eq:free-weak-new}
\int_0^T \bigl\langle e^{-itD}\psi_0,\,-i\zeta'(t)\phi-\zeta(t)D\phi\bigr\rangle_{L^2}\,dt=0.
\end{equation}
By density of $L^2(G;\mathbb C^2)$ in $H_D^s(G)$ and boundedness of $e^{-itD}$ on $H_D^s(G)$,
\eqref{eq:free-weak-new} extends to all $\psi_0\in H_D^s(G)$.

Choose an extension $F^{\mathrm{ext}}\in X^{s,b'-1}$ of $F$ to $\mathbb R$.
Since the dense core $\mathcal S=\bigcup_{M\in\mathbb N} C_c^\infty(\mathbb R;\mathrm{Dom}(\langle D\rangle^M))$ is dense in $X^{s,b'-1}$,
there exist $G_n\in\mathcal S$ such that $G_n\to F^{\mathrm{ext}}$ in $X^{s,b'-1}$.
Set $F_n=G_n|_{[0,T]}$ and
\[
W_n(t)=\int_0^t e^{-i(t-\tau)D}F_n(\tau)\,d\tau.
\]
Then $F_n\to F$ in $X_T^{s,b'-1}$ by definition of the restriction norm.
Moreover, each $G_n$ is smooth in $t$ with values in $\mathrm{Dom}(D)$, so differentiation under the integral sign gives
$W_n\in C^\infty([0,T];\mathrm{Dom}(D))$ and, for every $t\in(0,T)$,
\[
\frac{d}{dt}\langle W_n(t),\phi\rangle_{L^2}
=\langle F_n(t),\phi\rangle_{L^2}+\langle W_n(t),-iD\phi\rangle_{L^2}.
\]
Multiplying by $\zeta(t)$ and integrating by parts in $t$ gives
\begin{equation}\label{eq:duh-weak-new}
\int_0^T \bigl\langle iW_n(t),-i\zeta'(t)\phi-\zeta(t)D\phi\bigr\rangle_{L^2}\,dt
+\int_0^T \langle F_n(t),\zeta(t)\phi\rangle_{L^2}\,dt=0.
\end{equation}
By Lemma~\ref{lem:duhamel}, $W_n\to W$ in $X_T^{s,b}$, hence $W_n\to W$ in $C([0,T];H_D^s(G))$ by Lemma~\ref{lem:embed},
so we can pass to the limit in the first integral in \eqref{eq:duh-weak-new}.
Moreover, since $\zeta\phi$ is smooth and compactly supported in time, the second term converges to
$\int_0^T \langle F(t),\zeta(t)\phi\rangle_{L^2}\,dt$ in the distributional sense.
Letting $n\to\infty$ in \eqref{eq:duh-weak-new} yields
\begin{equation}\label{eq:duh-weak-limit}
\int_0^T \bigl\langle iW(t),-i\zeta'(t)\phi-\zeta(t)D\phi\bigr\rangle_{L^2}\,dt
+\int_0^T \langle F(t),\zeta(t)\phi\rangle_{L^2}\,dt=0.
\end{equation}

Adding \eqref{eq:free-weak-new} and \eqref{eq:duh-weak-limit} and using $\psi=e^{-itD}\psi_0+iW$ gives \eqref{eq:weak-form-mild},
which is exactly the distributional formulation of $i\partial_t\psi=D\psi-\mathcal N(\psi)$ in
$\mathcal D'((0,T);H_D^{s-1}(G))$.
\end{proof}

\begin{proof}[Proof of Theorem~\ref{thm:lwp}]
Let $s>-\frac18$. Fix $b\in\left(\frac12,1\right)$ and choose $b'\in\left(0,\frac12\right)$ with $b'<b$.
Let $b_0$ satisfy the assumptions of Proposition~\ref{prop:bilinear-mixed}.
For $T\in(0,1]$ define, for $\psi\in X_T^{s,b}$,
\[
\mathcal T(\psi)(t)=e^{-itD}\psi_0+i\int_0^t e^{-i(t-\tau)D}\,\mathcal N(\psi(\tau))\,d\tau.
\]
By Lemma~\ref{lem:free} and Lemma~\ref{lem:duhamel}, there exists $C>0$ such that
\begin{equation}\label{eq:T-bound}
\|\mathcal T(\psi)\|_{X_T^{s,b}}
\le C\|\psi_0\|_{H_D^s(G)}+C\,T^{1-b+b'}\|\mathcal N(\psi)\|_{X_T^{s,b'-1}}.
\end{equation}
By Lemma~\ref{lem:nonlinear},
\begin{equation}\label{eq:N-bound}
\|\mathcal N(\psi)\|_{X_T^{s,b'-1}}
\le C_{\mathcal B}\,C_N\,\|\psi\|_{X_T^{s,b}}^{2}.
\end{equation}
Set
\[
R=2C\|\psi_0\|_{H_D^s(G)},
\qquad
B_R=\{\psi\in X_T^{s,b}:\ \|\psi\|_{X_T^{s,b}}\le R\},
\qquad
C_\ast=C\,C_{\mathcal B}\,C_N.
\]
Combining \eqref{eq:T-bound}--\eqref{eq:N-bound}, for $\psi\in B_R$ we get
\[
\|\mathcal T(\psi)\|_{X_T^{s,b}}
\le \frac{R}{2}+C_\ast\,T^{1-b+b'}R^{2}.
\]
Choose $T\in(0,1]$ such that
\[
C_\ast\,T^{1-b+b'}R\le \frac14.
\]
Then $C_\ast\,T^{1-b+b'}R^{2}\le \frac{R}{4}$ and hence $\mathcal T$ maps $B_R$ into itself.

For $\psi_1,\psi_2\in B_R$, Lemma~\ref{lem:duhamel} and \eqref{eq:quad-lip} in Lemma~\ref{lem:nonlinear} yield
\[
\begin{aligned}
 \|\mathcal T(\psi_1)-\mathcal T(\psi_2)\|_{X_T^{s,b}}
&\le C\,T^{1-b+b'}\|\mathcal N(\psi_1)-\mathcal N(\psi_2)\|_{X_T^{s,b'-1}}\\
&\le C_\ast\,T^{1-b+b'}(\|\psi_1\|_{X_T^{s,b}}+\|\psi_2\|_{X_T^{s,b}})\|\psi_1-\psi_2\|_{X_T^{s,b}}.
\end{aligned}
\]
Since $\psi_1,\psi_2\in B_R$, this gives
\[
\|\mathcal T(\psi_1)-\mathcal T(\psi_2)\|_{X_T^{s,b}}
\le 2C_\ast\,T^{1-b+b'}R\,\|\psi_1-\psi_2\|_{X_T^{s,b}}
\le \frac12\,\|\psi_1-\psi_2\|_{X_T^{s,b}}.
\]
Thus $\mathcal T$ is a contraction on $B_R$, and there exists a unique fixed point $\psi\in B_R$.

By Lemma~\ref{lem:embed}, $\psi\in C([0,T];H_D^s(G))$.
Moreover, the map
\[
\Phi(t)=e^{-itD}\psi_0+i\int_0^t e^{-i(t-\tau)D}\,\mathcal N(\psi(\tau))\,d\tau
\]
belongs to $X_T^{s,b}\hookrightarrow C([0,T];H_D^s(G))$ by Lemma~\ref{lem:free}, Lemma~\ref{lem:duhamel}, and Lemma~\ref{lem:nonlinear}.
Since $\psi=\mathcal T(\psi)$ in $X_T^{s,b}$ and both sides are continuous $H_D^s(G)$-valued functions on $[0,T]$,
it follows that $\psi(t)=\Phi(t)$ in $H_D^s(G)$ for every $t\in[0,T]$.
Therefore Lemma~\ref{lem:mild-weak} applies and shows that $\psi$ is a distributional solution of \eqref{eq:NLDE} on $(0,T)$.

Uniqueness in $X_T^{s,b}$ implies uniqueness in $C([0,T];H_D^s(G))$ since $X_T^{s,b}\hookrightarrow C([0,T];H_D^s(G))$.
Local Lipschitz dependence on the data follows by applying the same contraction estimate to the difference of the fixed points
corresponding to $\psi_0$ and $\widetilde\psi_0$ on the common time interval determined by
$\max\{\|\psi_0\|_{H_D^s(G)},\|\widetilde\psi_0\|_{H_D^s(G)}\}$.

Let $[0,T^\ast)$ be the maximal forward lifespan of $\psi$ in $C([0,T);H_D^s(G))$ and assume that
\[
\sup_{t\uparrow T^\ast}\|\psi(t)\|_{H_D^s(G)}<\infty.
\]
Set
\[
M=\sup_{t\in[0,T^\ast)}\|\psi(t)\|_{H_D^s(G)}<\infty.
\]
Fix $t_0\in[0,T^\ast)$ and define $u(t)=\psi(t_0+t)$ for $t\ge 0$. Then $u$ solves \eqref{eq:NLDE} with initial datum
$u(0)=\psi(t_0)$.
The above contraction argument applied to $u(0)$ yields a local solution on an interval $[0,\delta]$ with
\[
\delta=\delta\bigl(\|\psi(t_0)\|_{H_D^s(G)}\bigr).
\]
Since $\|\psi(t_0)\|_{H_D^s(G)}\le M$, there exists $\delta_0=\delta_0(M)>0$ such that one can take $\delta\ge \delta_0$
for every $t_0\in[0,T^\ast)$.
Choosing $t_0$ so that $T^\ast-t_0<\delta_0$, we obtain a solution on $[t_0,t_0+\delta_0]$, hence on an interval strictly beyond $T^\ast$.
This contradicts the maximality of $T^\ast$, and therefore $\psi$ extends past $T^\ast$.
\end{proof}

\section*{Acknowledgments}
We would like to thank the anonymous referee for his/her careful readings of our manuscript and the useful comments. 

\medskip
{\bf Funding:} This work is supported by National Natural Science Foundation of China (12301145, 12261107, 12561020) and Yunnan Fundamental Research Projects (202301AU070144, 202401AU070123).

\medskip
{\bf Author Contributions:} All the authors wrote the main manuscript text together and these authors contributed equally to this work.

\medskip
{\bf Data availability:}  Data sharing is not applicable to this article as no new data were created or analyzed in this study.

\medskip
{\bf Conflict of Interests:} The authors declare that there is no conflict of interest.

\bibliographystyle{plain}
\bibliography{reference}
\end{document}